\documentclass{amsproc}
\usepackage{euscript, graphicx, epstopdf}
\usepackage{cases}
\usepackage{mathrsfs}
\usepackage{bbm}
\usepackage{amssymb}
\usepackage{txfonts}
\usepackage{amscd}
\usepackage{amsfonts,latexsym,amsmath,amsxtra,mathdots,amssymb,latexsym,mathabx,mathtools}
\usepackage[all,cmtip]{xy}
\usepackage{color}
\usepackage{colordvi}
\usepackage{multicol}
\usepackage{hyperref}
	\hypersetup{pdfcreator=XeLaTeX with hyperref}
\usepackage{tikz}
\usepackage{float}
\usepackage{setspace, floatflt}

\usepackage{tensor}

\allowdisplaybreaks

\newcommand{\BC}{{\mathbb {C}}} 
 \newcommand{\BF}{{\mathbb {F}}}

 \newcommand{\BR}{{\mathbb {R}}}

\newcommand{\BW}{{\mathbb {W}}}

\newcommand{\GL}{{\mathrm {GL}}}

 \newcommand{\Tr}{{\mathrm{Tr}}}

\newcommand{\Gal}{\mathrm{Gal}} \newcommand{\N}{\mathrm{N}}
\newcommand{\res}{\mathrm{Res}}
\newcommand{\ind}{\mathrm{Ind}} \newcommand{\cind}{c\mbox{-}\mathrm{Ind}}
\newcommand{\llc}{\boldsymbol{l}}
\newcommand{\weil}{\mathcal{W}}

\newcommand{\ra}{\rightarrow}

\def\-{^{-1}}

\def\shskip{\hskip 0.5 pt}

\makeatletter
\g@addto@macro\normalsize{\setlength\abovedisplayskip{3pt}}
\makeatother

\makeatletter
\g@addto@macro\normalsize{\setlength\belowdisplayskip{3pt}}
\makeatother

\newcommand{\delete}[1]{}

\theoremstyle{plain}

\newtheorem{thm}{Theorem}[section] 
\newtheorem{lem}[thm]{Lemma}  \newtheorem{prop}[thm]{Proposition}

\newtheorem {rem}[thm]{Remark}

\numberwithin{equation}{section}

\begin{document}
	\setcounter{section}{-1}
	
	\title{Gamma factors of level zero supercuspidal representations}

	\author{Chang Yang}
	\address{Key Laboratory of High Performance Computing and Stochastic Information Processing (HPCSIP)\\ College of Mathematics and Statistics \\ Hunan Normal University \\Changsha,  410081\\China}
	\email{cyang@hunnu.edu.cn}


	\begin{abstract}
		We give an explicit formula for the twisted gamma factor for a pair of irreducible supercuspidal representations of level zero. We also obtain an explicit formula for the unramified base change of level zero supercuspidal representations.
	\end{abstract}
	
	\maketitle
	
	\section{Introduction}
	
	\subsection{Main result and Motivation}
	Throughout this note, let $F$ be a $p$-adic field with residue field of $q$ elements. Denote by $\mathfrak{o}$ and $\mathfrak{p}$ the ring of integers of $F$ and the maximal ideal of $\mathfrak{o}$.
	
	 Let $\pi$ and $\tau$ be irreducible smooth representations of $\GL_n(F)$ and $\GL_m(F)$. Let $\psi$ be a nontrivial character of the additive group of $F$. The gamma factor $\gamma (s,\pi \times \tau,\psi)$ was first introduced in \cite{J-PS-S-Rankin+Selberg}, along with an $L$-function $L(s,\pi \times \tau)$ and a local constant $\varepsilon(s,\pi \times \tau,\psi)$.  
	
	It is desirable to have explicit formulae for these gamma factors. By \cite{J-PS-S-Rankin+Selberg}, the computation can be reduced to the case of supercuspidal representations. When $\tau=1_F$ is the trivial representation of $\GL_1(F)$, then $\gamma(s,\pi \times 1_F,\psi)$ is nothing but the gamma factor $\gamma(s,\pi,\psi)$ in the sense of Godement-Jacquet \cite{Godement-Jacquet}. There are explicit formulae for $\gamma(s,\pi,\psi)$ in terms of ``non abelian Gauss sums'' (See \cite{Bushnell-orders+Gauss+sums}, \cite{Bushnell-Frohlich-Gauss+sum}; also \cite{B-H-GL2} for the $\GL_2$ case). On the other hand, in the twisted case, one has only an explicit formula for the conductor of the local constants $\varepsilon(s,\pi \times \tau,\psi)$ (See \cite{B-H-K-Conductor}). 
	
	In this short note, we confine ourselves to irreducible supercuspidal representations of level zero. A smooth representation $\pi$ of $\GL_n(F)$ is called of level zero if it contains a nontrivial fixed vector for the subgroup $1+\mathfrak{p}\mathrm{M}_n(\mathfrak{o})$ of $\GL_n(\mathfrak{o})$. Our main result is
     \begin{thm}
     	Let $\pi$ resp. $\tau$ be an irreducible supercuspidal representation of level zero of $\GL_n(F)$ resp. $\GL_m(F)$, with $n>m$. Let $\psi$ be a character of $F$ of levle one, that is, $\psi$ is trivial on $\mathfrak{p}$ but not on $\mathfrak{o}$. Then
     	\begin{align}\label{formula::Introduction-Main}
     	\gamma (s,\pi \times \tau,\psi) = (-1)^{nm-(n,m)}q^{-nm/2}\prod_{i=0}^{(n,m)-1} G(\shskip \widetilde{\eta_{\pi}}^{q^i}\circ \N_{[n,m]:n}\cdot \widetilde{\eta_{\tau}}\circ \N_{[n,m]:m},\tilde{\psi}).
     	\end{align}
     Here
     \begin{itemize}
     	\item[$\bullet$] $(n,m)$ is the greatest common divisor of $n$ and $m$;
     	\item[$\bullet$] $\widetilde{\eta_{\pi}}$ resp. $\widetilde{\eta_{\tau}}$ is the regular character of $\BF_{q^n}^{\times}$ resp. $\BF_{q^m}^{\times}$ that corresponds to $\tilde{\pi}$ resp. $\tilde{\tau}$ via Green's parameterization (see \S \ref{section::n*1}, \eqref{diagram::Green}), while $\tilde{\pi}$ resp. $\tilde{\tau}$ is the cuspidal representation of $\GL_n(\BF_q)$ resp. $\GL_m(\BF_q)$ which is uniquely determined by $\pi$ resp. $\tau$ (see \S \ref{section::gamma-finite+local}, \eqref{diagram::level 0});
     	\item[$\bullet$] $\N_{[n,m]:n}$ resp. $\N_{[n,m]:m}$ is the norm map from $\BF_{q^{[n,m]}}^{\times}$ to $\BF_{q^n}^{\times}$ resp. $\BF_{q^m}^{\times}$;
     	\item[$\bullet$] $\tilde{\psi}$ is the character on $\kappa_F \cong \BF_q$ which is induced by $\psi$, $\ \tilde{\psi}(\bar{x})=\psi(x)\text{ for all }x\in \mathfrak{o}$. 
      \end{itemize}
     and 
      \begin{align}\label{defn::Gauss sum}
     G(\beta,\varphi) := \sum_{a \in \BF_{q^N}^{\times}} \beta(a) \varphi (\Tr \shskip a^{-1}) = \sum_{a \in \BF_{q^N}^{\times}} \beta^{-1}(a) \varphi (\Tr\shskip a)
     \end{align}
     is the Gauss sum for a multiplicative character $\beta$ of $\BF_{q^N}^{\times}$ and an additive character $\varphi$ of $\BF_q$, where $\Tr$ is the trace map from $\BF_{q^N} \ra \BF_q$.
     \end{thm}
 
	The motivation of our work arises from a question on gamma factors over finite fields that was proposed in \cite[Conjecture 2.2]{Nien-Zhang}. A formula similar to \eqref{formula::Introduction-Main} for gamma factors over finite fields is given in Theorem \ref{theorem::Main-Gamma-finite}. Let us add some words on the strategy of our method. We first use a result in \cite{B-H-K-LLC+Conductor} to reduce the general case to the case $m=1$. The main point of our work is then to determine the base change, in the sense of \cite{Arthur-Clozel}, of a level zero supercuspidal representation. We use the explicit local Langlands correspondence for level zero representations, developped in \cite{B-H-Explicit+level+0}, to transfer the question to the Galois side. When $m=1$, the computation is a consequence of a result on gamma factors over finite fields \cite{Nien-n*1} and a connection formula between gamma factors for level zero supercuspidal represenations and gamma factors over finite fields \cite{Nien-Zhang}.

    \subsection{Notations and Conventions}
   
    (1)\ For a smooth representation $\pi$ of $\GL_n(F)$, denote by $\pi^{\vee}$ the smooth dual of $\pi$. If $\pi$ is irreducible, denote by $\omega_{\pi}$ the central character of $\pi$. For representations $\pi_i$ of $\GL_{n_i}(F)$, $i=1,\cdots,r$, denote by $\pi_1 \times \cdots \times \pi_r$ the representation of $\GL_{n_1+\cdots+n_r}(F)$ obtained from $\pi_1\otimes \cdots \otimes\pi_r$ by normalized parabolic induction. For a representation $\pi$ and a character $\chi$ of $\GL_n(F)$, denote $\chi \pi$ to be the representation on the space of $\pi$ given by $(\chi \pi)(g)=\chi(g)\pi(g)$. Let $|\cdot|$ denote the normalized absolute value on $F$ and $\nu$ denote the character $|\cdot|\circ \det$ of $\GL_n(F)$.
    
    (2)\ Denote by $U$ the group of units in $\mathfrak{o}$, by $\kappa = \mathfrak{o}/\mathfrak{p}$ the residue field of $F$. We fix a uniformizer $\varpi$ in $\mathfrak{p}$.

     If $E/F$ is a finite field extension, we use the analogous notations $\mathfrak{o}_E$, $\mathfrak{p}_E$, $\mathrm{U}_E$, etc. The norm map $E^{\times} \ra F^{\times}$ is denoted $\N_{E/F}$, and the trace $E \ra F$ is $\Tr_{E/F}$.
     
    We fix a separable algebraic closure $\bar{F}$ of $F$. All finite extension of $F$ are supposed to be contained in $\bar{F}$. Hence an unramifed extension over $F$ of a fixed degree is unique. For $E/F$ a finite unramified extension, we still use $\varpi$ as a uniformizer in $\mathfrak{p}_E$.
    
    For a finite field extension $K/F$, denote by $\weil_K$ the Weil group of $K$. We will view $\weil_K$ as a subgroup of $\weil_F$ as in \cite[\S 28.5]{B-H-GL2}. Let $\boldsymbol{a}_F \colon \weil_F \ra F$ denote the Artin reciprocity map (\cite[\S 29.1]{B-H-GL2}). We fix a geometric Frobenius element $\Phi$ in $\weil_F$ and a arithmetic Frobenius element $\phi = \Phi^{-1}$.
    
    If $E/F$ is a finite extension, we write $\ind_{E/F}$ rather than $\ind_{\mathcal{W}_E}^{\weil_F}$ for the functor of smooth induction from $\weil_E$ to $\weil_F$. Also, if $\rho$ is a smooth representation of $\weil_F$, we put $\res_{E/F}\shskip \rho = \rho |_{\weil_E}$.

  \section{Preliminaries}
    \subsection{Gamma factors for Weil representations} 
    
     Let $\mathcal{G}^{ss}_n(F)$ denote the set of isomorphism classes of semisimple smooth representations of $\weil_F$ of dimension $n$ and write $\mathcal{G}^{ss}(F) = \bigcup_{n \geqslant 1} \mathcal{G}^{ss}_n(F)$\footnote{In this note, we do not need the nilpotent part in Weil-Deligne representations.}. For $\sigma \in \mathcal{G}^{ss}(F)$ and a nontrivial character $\psi_F$ of $F$, let $L(s,\sigma)$ be the Artin $L$-function and $\varepsilon(s,\sigma,\psi_F)$ be the local constant defined by Langlands and Deligne (see \cite{Deligne-Local+Constants}). Denote by $\sigma^{\vee}$ the contragradient representation of $\sigma$. Similar to the definition of local gamma factors in \cite{J-PS-S-Rankin+Selberg}, we define
   \begin{align}\label{formula::defn-gamma}
   \gamma(s,\sigma,\psi_F) = \varepsilon(s,\sigma,\psi_F) \frac{L(1-s,\sigma^{\vee})}{L(s,\sigma)}.
   \end{align}
   
   Suppose that $\sigma_1,\sigma_2 \in \mathcal{G}^{ss}(F)$. By the additive properties of $L$ and $\varepsilon$ (\cite[\S 29.3,29.4]{B-H-GL2}), we have
   \begin{align}
   \gamma(s,\sigma_1\oplus \sigma_2,\psi_F)& = \gamma(s,\sigma_1,\psi_F) \gamma(s,\sigma_2,\psi_F). \label{formula::gamma--additive}
   \end{align}
   
   Let $K/F$ be a finite field extension and $\rho \in \mathcal{G}^{ss}_n(K)$. Note that $(\ind_{K/F}\shskip \rho)^{\vee} \cong \ind_{K/F}\shskip \rho^{\vee} $. By the inductive properties of $L$ and $\varepsilon$ (\cite[\S 29.3,29.4]{B-H-GL2}), we have 
   \begin{align}
   \frac{\gamma(s,\ind_{K/F}\shskip \rho,\psi_F)}{\gamma(s,\rho,\psi_K)} &= \frac{\gamma(s,\ind_{K/F}\shskip 1_K,\psi_F)^n}{\gamma(s,1_K,\psi_K)^n}, \label{formula::gamma--inductive}
   \end{align}
   where $1_K$ stands for the trivial representation of $\weil_K$ and $\psi_K=\psi_F\circ \Tr_{K/F}$.
   \begin{rem}
   	The functions $L$ and $\varepsilon$ can be defined for Weil-Deligne representations of $F$ (see \cite[\S 31.3]{B-H-GL2}), so is $\gamma$ by \eqref{formula::defn-gamma}. It turns out that $\gamma(s,(\sigma,\mathrm{n}),\psi_F) = \gamma (s,\sigma,\psi_F)$, where we denote by $(\sigma,\mathrm{n})$ a Weil-Deligne representation and $\mathrm{n}$ is the nilpotent part.
   \end{rem}
   
      \subsection{Local Langlands correspondence for $\GL_n(F)$}
      Let $\mathcal{A}_n(F)$ be the set of isomorphism classes of irreducible smooth representations of $\GL_n(F)$ and $\mathcal{A}^0_n(F)$ its subset consisting of supercuspidal representations. Let $\mathcal{G}_n(F)$ be the set of isomorphism classes of semisimple Weil-Deligne representations of $F$ of dimension $n$. Denote by $\mathcal{G}^0_n(F)$ the set of isomorphism classes of irreducible smooth representations of $\weil_F$ of dimension $n$. We identify $\mathcal{G}^0_n(F)$ as a subset of $\mathcal{G}_n(F)$ by letting the nilpotent part be $0$.
   
   The Langlands correspondence for $\GL_n$ over $F$, proved by Harris-Taylor \cite{Harris-Taylor} and by Henniart \cite{Henniart-proof+of+LLC}, is the unique collection of bijections 
   \begin{align}\label{formula::LLC}
   \llc_n^F \colon \mathcal{A}_n ( F ) \ra \mathcal{G}_n ( F )
   \end{align}
   such that the map $\llc_1^F$ is given by the local class field theory, and that for $\pi \in \mathcal{A}_n(F)$, $\pi' \in \mathcal{A}_{n'} (F)$, we have
   \begin{align*}
   L ( s, \pi \times \pi') & =  L (s, \llc_n^F(\pi) \otimes \llc_{n'}^F ( \pi') ), \\
   \varepsilon ( s, \pi \times \pi' , \psi)  &=  \varepsilon (s, \llc_n^F(\pi) \otimes \llc_{n'}^F ( \pi'),\psi). \label{formula::LLC-epsion}
   \end{align*}
   Hence
   \begin{align}
   \gamma(s,\pi \times \pi',\psi) = \gamma(s, \llc_n^F(\pi)\otimes \llc_n^F(\pi'),\psi).
   \end{align}
   
   The maps $\{\llc_n^F\}$ are first constructed for $\mathcal{A}^0_n(F)$ and then extended to $\mathcal{A}_n (F)$, using the classification of Langlands, Bernstein and Zelevinsky. We recall briefly this process (see \cite[Chapter 2]{Henniart-Characterization+LLC}). According to \cite[9.3]{Zelevinsky-II}, each essentialy square-integrable representation $\pi$ is of the form $\mathrm{St}_k(\rho)$, for some $k \geq 1$ and $\rho \in \mathcal{A}^0_n(F)$, where $\mathrm{St}_k(\rho)$ denotes the unique irreducible quotient of $\rho \times \nu\rho \times \cdots \times \nu^{k-1}\rho$. We have 
   \begin{align}
   \llc_n^F (\mathrm{St}_k (\rho)) = \llc_n^F(\rho) \otimes \mathrm{Sp}_k,
   \end{align}
   where $\mathrm{Sp}_k$ denotes the special Weil-Deligne representation of $\weil_F$ of dimension $k$ (\cite[\S 31.1]{B-H-GL2}). Denote by $\mathcal{A}^{\square}_n(F)$ the isomorphism classes of essentially square-integrable representations. For every $\pi \in \mathcal{A}_n^{\square}(F)$, there is a unique $\alpha(\pi ) \in \BR$ such that $\nu^{-\alpha}\pi$ is unitary and square-integrable. Let $r \geq 1$ and $\pi_1,\cdots,\pi_r$ be elements of $\mathcal{A}^{\square}_n(F)$. Suppose that
   \begin{align}\label{Condition}
   \alpha(\pi_1) \geq \cdots \geq \alpha(\pi_r). \tag{$\star$}
   \end{align}
   Then the representation $\pi_1 \times \cdots \times \pi_r$ has a unique irreducible quotient $J(\pi_1,\cdots,\pi_r)$. 
   By the Langlands classification (\cite{Silberger-Langlands+Classification}), each $\pi\in \mathcal{A}_n(F)$ is of the form $J(\pi_1,\cdots,\pi_r)$ for some $r \geq 1$ and $\pi_1,\cdots,\pi_r \in \mathcal{A}^{\square}_n(F)$. The isomorphism class of $J(\pi_1,\cdots,\pi_r)$ does not depend on the order of $\pi_i$ provided that the conditon (\ref{Condition}) is satisfied. We have
   \begin{align}\label{formula::LLC-additive}
   \llc_n^F (J(\pi_1,\cdots,\pi_r)) = \llc_{n_1}^F(\pi_1) \oplus \cdots \oplus \llc_{n_r}^F(\pi_r).
   \end{align}
   
    \section{Base change and automorphic induction}\label{section::BC+AI}

    A key ingredient in our method is to use a formula on the twisted local gamma factors which compares base change and automorphic induction.
   		
    Let $K/F$ be a cyclic extension. Recall that, in \cite[\S 1.6]{Arthur-Clozel}, base change associates to every $\pi \in \mathcal{A}^0_n(F)$ an irreducible smooth representation $\pi_{K/F}$ of $\GL_n(K)$\footnote{The definition of $\pi_{K/F}$ was extended to all $\pi \in \mathcal{A}_n(F)$ in \cite[pp. 59-60]{Arthur-Clozel}, but we do not need this fact in our context.}. Via the Langlands correspondence, the base change from $F$ to $K$ corresponds to the restriction to $\weil_K$ from representations of $\mathcal{W}_F$ (\cite[\S 3.1]{B-H-K-LLC+Conductor}). So, we have
   \begin{align}\label{formula::base change and LLC}
   \res_{K/F}(\boldsymbol{l}^F_n (\pi)) =  \boldsymbol{l}^K_n ( \pi_{K/F} ).
   \end{align}
   
    Assume that $K/F$ is cyclic of degree $d$. According to \cite{Henniart-Herb-AI}, automorphic induction sends any generic $\rho$ in $\mathcal{A}_n(K)$ to a unique $\rho^{K/F} \in \mathcal{A}_{nd}(F)$ (see also \cite[\S 2.5]{B-H-K-LLC+Conductor}). The automorphic induction from $K$ to $F$ corresponds, again via the Langlands correspondence, to the induction to $\weil_F$ from representations of $\weil_K$, i.e., 
    \begin{align}\label{formula::automorphic induction}
    \ind_{K/F}(\llc_n^K (\rho)) = \llc_{nd}^F(\, \rho^{K/F}),
    \end{align}
    where $\rho$ is generic (\cite[\S 3.9]{B-H-K-LLC+Conductor}).

    The following result can be found in \cite[A.8]{B-H-K-LLC+Conductor}. We supply here a ``proof'' different than that in loc.cit.
    \begin{prop}
    	Assume $K/F$ to be a cyclic extension. Let $\pi \in \mathcal{A}^0_n(F)$ and $\rho \in \mathcal{A}^0_{n'}(K)$ with $\rho$ generic. Let $\psi_F$ be a nontrivial character of $F$. Then
    	
    	\begin{align}\label{formula::gamma-BC + AI}
    	\frac{\gamma (s, \pi \times \rho^{K/F},\psi_F)}{\prod_{\chiup \in \varXi} \gamma (s,\chi,\psi_F)^{nn'}}  =   \frac{\gamma (s,\pi_{K/F} \times \rho ,\psi_K)}{ \gamma (s,1_K,\psi_K)^{nn'}},
    	\end{align}
    	where $\varXi$ is the group of characters of $F^{\times}$ that is trivial on $\N_{K/F}(K^{\times})$ and $\psi_K := \psi_F \circ \mathrm{Tr}_{K/F}$.
    \end{prop}
    
    \begin{proof}
    	We omit the subscripts in the $\llc_n^F$.
    	\begin{align}
    	\gamma (s,\pi \times\rho^{K/F},\psi_F) & = \gamma (s,\llc(\pi)\otimes \llc(\rho^{K/F}),\psi_F)   \\
    	&=\gamma(s,\llc(\pi)\otimes \ind_{K/F} \llc(\rho),\psi_F) \nonumber \\
    	&=\gamma (s,\ind_{K/F}(\mathrm{Res}_{K/F}\llc(\pi) \otimes \llc(\rho )),\psi_F). \nonumber
    	\end{align}
    	In view of \eqref{formula::gamma--inductive}, we have
    	\begin{align}
    	&\gamma (s,\ind_{K/F}(\mathrm{Res}_{K/F}\llc(\pi) \otimes \llc(\rho )),\psi_F)   \nonumber \\&=   \gamma(s,\mathrm{Res}_{K/F}\llc(\pi) \otimes \llc(\rho),\psi_K) \cdot \frac{\gamma (s,\ind_{K/F}1_K,\psi_F)^{nn'} }{\gamma(s, 1_K,\psi_K)^{nn'}} \\
    	&=\gamma (s,\pi_{K/F} \times \rho ,\psi_K)\frac{\gamma (s,\ind_{K/F}1_K,\psi_F)^{nn'} }{\gamma(s, 1_K,\psi_K)^{nn'}}.\nonumber
    	\end{align}
    	Note that, by local class field theory, $\varXi$ corresponds to the group of characters of $\weil_F$ that is trivial on $\weil_K$. The equality \eqref{formula::gamma-BC + AI} follows then from the local Langlands correspondence \eqref{formula::LLC-epsion} and the simple fact that
    	\begin{align*}
    	\ind_H^G \shskip 1_H = \mathop{\oplus}_{\chi \in \widehat{G/H} }\shskip \shskip \chi \shskip,
    	\end{align*}
    	where $H$ is a normal subgroup of $G$ of finite index such that $G/H$ is abelian, $1_H$ is the trivial representation of $H$ and $\widehat{G/H}$ is the group of characters of $G/H$.
    \end{proof}
    \begin{rem}
    	Strictly speaking, this is not really a proof, as the formula \eqref{formula::gamma-BC + AI} predates and is used to establish the local Langlands correspondence for $\GL_n(F)$, while we use the Langlands correspondence in the ``proof''. The original proof in \cite{B-H-K-LLC+Conductor} involves a global method.
    \end{rem}
    
    \section{Gamma factors over finite fields}\label{section::gamma factor-finite field}
    
    \subsection{Baisc facts}
    In her thesis \cite{Roditty}, Roditty considered a finite field analogue of Rankin-Selberg convolution in the same way as in \cite{J-PS-S-Rankin+Selberg}. There she defined the gamma factor (a complex number in this case, instead of a function!) as the ratio appeared in certain functional equations. 
    
    We fix a nontrivial character $\psi$ of $\BF_q$ in this subsection. Let $\mathrm{U}_n(\BF_q)$ be the group of upper triangular unipotent matrices in $\GL_n(\BF_q)$. Let 
    \begin{align*}
    \psi_n(u) = \psi (\sum\limits_{i=1}^{n-1} u_{i,\shskip i+1}),\quad \text{for }u=(u_{ij})\in \mathrm{U}_n(\BF_q).
    \end{align*}
    If $\pi$ is a generic representation of $\GL_n(\BF_q)$, denote by $\mathcal{W}(\pi,\psi_n)$ the Whittaker model of $\pi$ with respect to $\psi_n$. We refer to \cite{Nien-Jacquet+conjecture} for the undefined notations in this subsection.
    
    \begin{thm}[\cite{Roditty}, Theorem 5.1, 5.4 or \cite{Nien-Jacquet+conjecture}, Theorem 2.10]
    	Let $\pi$ be an irreducible cuspidal representation of $\GL_n(\BF_q)$ and $\tau$ an irreducible generic representation of $\GL_m(\BF_q)$, with $n > m$. Then there exists a complex number $\gamma (\pi \times \tau,\psi)$ such that
    	\begin{align*}
    	& \gamma(\pi \times \tau,\psi) q^{km} \sum_{ g \in \mathrm{U}_m \backslash \GL_m(\BF_q)}\sum_{x \in \mathrm{M}_{n-m-1-k,m}} W_{\pi}\big( \left( \begin{matrix}
    		 g & 0 & 0 \\ x & I_{n-m-1-k} & 0 \\ 0 & 0& I_{k+1}
    		\end{matrix} \right) \big) W_{\tau} ( g )  \\
    	&=    \sum_{ g \in \mathrm{U}_m \backslash \GL(\BF_q)} \sum_{y \in \mathrm{M}_{m,k}} W_{\pi} \big( \left( \begin{matrix}
    		0 & I_{n-m-k} & 0 \\ 0 & 0 & I_k \\ g & 0 & y
    		\end{matrix} \right)  \big) W_{\tau} ( g ),	
    	\end{align*}
    	for all $0 \leq k \leq n-m-1$, $W_{\pi} \in \mathcal{W}(\pi, \psi_n)$ and $W_{\tau} \in \mathcal{W}(\tau, \psi_m^{-1})$.
    \end{thm}
    
    There is a distinguished element in the Whittaker model that can be used to compute the gamma factors.
    \begin{prop}[\cite{Gelfand-finite}, Proposition 4.5]\label{prop::Bessel}
    	Let $\pi$ be an irreducible generic representation of $\GL_n(\BF_q)$ with $\chi_{\pi}$ its (trace) character. The functon
    	\begin{align}\label{formula::Bessel}
    	J_{\pi,\psi_n} (g) = |\mathrm{U}_n(\BF_q)|^{-1} \sum_{ u \in \mathrm{U}_n(\BF_q)} \chi_{\pi}(gu) \psi_n(u^{-1}) 
    	\end{align}
    	lies in $\mathcal{W}(\pi,\psi_n)$ and satisfies
    	\begin{align*}
    	J_{\pi,\psi_n}(u_1 g u_2) = \psi_n(u_1u_2)J_{\pi,\psi_n}(g),\ u_i \in \mathrm{U}_n(\BF_q) \quad \text{and}\quad J_{\pi,\psi_n}(I_n)=1.
    	\end{align*}
    \end{prop}  
    The function $J_{\pi,\psi_n}$ is called the (normalized) Bessel function of $\pi$ with respect to $\psi_n$. In terms of Bessel functions, we have the following
    \begin{prop}[\cite{Roditty}, Lemma 6.1.4] \label{prop::gamma-Bessel}
    	Let $\pi$ be an irreducible cuspidal representation of $\GL_n(\BF_q)$ and $\tau$ an irreducible generic representation of $\GL_m(\BF_q)$, $n>m$. Then
    	\begin{align}\label{formula::gamma-Bessel}
    	\gamma(\pi\times \tau, \psi)  =  \sum_{g \in \mathrm{U}_m \backslash \GL_m(\BF_q)} J_{\pi,\psi_n}\big( \left(  \begin{matrix}
    	0 & I_{n-m}  \\ g & 0
    	\end{matrix}\right) \big) J_{\tau, \psi_m^{-1} } (g).
    	\end{align}
    \end{prop}

    \subsection{$n \times 1$ gamma factors}\label{section::n*1}
    We recall Green's parameterization \cite{Green} for irreducible cuspidal representations of $\GL_n(\BF_q)$. A character $\eta$ of $\BF_{q^n}^{\times}$ is called $\BF_q$-regular if the conjugates $\eta^s$, $s\in \Gal(\BF_{q^n}/\BF_q)$ are distinct. Two characters $\eta_1$ and $\eta_2$ are called equivalent if $\eta_1 = \eta_2^{q^j}$ for some integer $j$, i.e., they are in the same Frobenius orbit. 
    
    \begin{rem}
    	A character $\eta$ is $\BF_q$-regular if and only if $\eta$ cannot factor through $\N_{\BF_{q^n} / \BF_{q^d}}$ for some subextension $\BF_{q^d}/\BF_q$.
    \end{rem}
   
     Let $\Lambda_n(\BF_q)$ denote the set of equivalence classes of $\BF_q$-regular characters of $\BF_{q^n}^{\times}$. Let $\mathcal{A}_n^0(\BF_q)$ denote the set of isomorphism classes of irreducible representations of $\GL_n(\BF_q)$.  Green's parameterization gives a bijection
    \begin{align}\label{diagram::Green}
     \mathcal{A}_n^0(\BF_q) & \longleftrightarrow \Lambda_n(\BF_q)  \\
      \pi  & \longleftrightarrow  \eta_{\pi}  \nonumber
    \end{align}
    The (trace) character $\chi_{\pi}$ of $\pi$ is connected with $\eta_{\pi}$ in an explicit way (see \cite[\S 6]{Gelfand-finite}). So, by \eqref{formula::Bessel}, the Bessel function $J_{\pi,\psi_n}$ can be expressed in terms of $\eta_{\pi}$. 
    
    In view of \eqref{formula::gamma-Bessel}, Nien computes the $n\times 1$ gamma factor as an ``abelian Gauss sum" using special values of Bessel functions. Denote by $\widehat{\BF_q^{\times}}$ the set of characters of $\BF_q^{\times}$.
    \begin{prop}[\cite{Nien-n*1}, Theorem 1.1]
    	Let $\pi$ be an irreducible cuspidal representation of $\GL_n(\BF_q)$, $n \geq 2$ and $\tau \in \widehat{\BF_q^{\times}}$. Then
    	\begin{align}\label{formula::Nien n*1}
    	\gamma ( \pi \times \tau ,\psi) = (-q^{-1} \tau (-1))^{n-1} G(\eta_{\pi}\cdot \tau\circ \N_{n:1},\psi).
    	\end{align}
    	where $\eta_{\pi}$ is the regular character of $\BF_{q^n}^{\times}$ which corresponds to $\pi$ by Green's parameterization; the Gauss sum $G$ is as defined in \eqref{defn::Gauss sum}.
    \end{prop}

    \subsection{Connection with local gamma factors over $p$-adic fields}\label{section::gamma-finite+local}
    
    We recall a result of Nien and Zhang which shows that gamma factors over finite fields and gamma factors for level zero supercuspidal representations over $p$-adic fields are closely related.
    
    A representation $\pi \in \mathcal{A}^0_n(F)$ is called of level zero if it contains a nontrivial fixed vector for the subgroup $1+\mathfrak{p}\mathrm{M}_n(\mathfrak{o})$ of $\GL_n(\mathfrak{o})$. According to \cite[Theorem 8.4.1]{Bushnell-Kutzko}, every level zero supercuspidal representation is of the form 
    \begin{align}
    \pi \cong \cind_{F^{\times}\GL_n(\mathfrak{o})}^{\GL_n(F)}\shskip \chi 
    \sigma,
    \end{align}
    where $\sigma$ is a representation of $\GL_n(\mathfrak{o})$ that is inflated from an irreducible cuspidal representation $\widetilde{\sigma}$ of $\GL_n(\BF_q)=\GL_n(\mathfrak{o}/\mathfrak{p})$, $\chi$ is a character of $F^{\times}$ such that $\chi |_{U_F} $ equals to the central character of $\sigma$ and $\cind$ is the compact induction. 
    
    The character $\chi$ and the representation $\widetilde{\sigma}$ are uniquely determined by $\pi$. Let $\mathcal{A}_n^0(F)_0$ denote the set of isomorphism classes of level zero supercuspidal representations of $\GL_n(F)$. Theorem 8.4.1 in \cite{Bushnell-Kutzko} then gives a bijection 
    \begin{align}\label{diagram::level 0}
     \mathcal{A}_n^0(F)_0  &\longleftrightarrow \BC^{\times} \times \mathcal{A}_n^0(\BF_q)  \\
     \pi & \longleftrightarrow (\omega_{\pi}(\varpi),\widetilde{\sigma}).  \nonumber
    \end{align}
    
    \underline{Convention}: For the rest of this note, if $\pi \in \mathcal{A}_n^0(F)_0$, we shall always denote by $\widetilde{\pi}$ the second component of the image of $\pi$ under the bijection \eqref{diagram::level 0}.
    
     \begin{thm}[\cite{Nien-Zhang}, Theorem 3.11]\label{thm::Nien-Zhang}
     	Let $\pi \in \mathcal{A}^0_n(F)_0$ and $\tau \in \mathcal{A}^0_m(F)_0$, $n > m$. Suppose that $\psi$ has level one. Then 
     	\begin{align}\label{formula::Connection finite-local}
     	 \gamma (s , \pi \times \tau, \psi )  = \omega_{\tau}(-1)^{n-1}   q^{m(n-m-1)/2} \gamma (\widetilde{\pi} \times \widetilde{\tau}, \widetilde{\psi}),
     	\end{align}
     	where $\widetilde{\psi}$ is the character of $\kappa_F \cong \BF_q$ that is induced by $\psi$.
     \end{thm}
    

  \section{Explicit Local Langlands Correspondence for level zero supercuspidal representations}
  To determine the base change of a representation as indicated in Section \ref{section::BC+AI}, we need an explicit knowledge of the local Langlands correspondence. The explicit correspondence for level zero representations of $\GL_n(F)$ was worked out by Bushnell and Henniart in \cite{B-H-Explicit+level+0}. We recall it here.
    \subsection{Admissible tame pairs}
    Recall that a tame pair over $F$ consists of a finite unramified field extension $E/F$ and a character $\theta$ of $E^{\times}$ that is trivial on $U_E^1=1+\mathfrak{p}_E$. A tame pair $(E/F,\theta)$ is called admissible if the conjugates $\theta^s$, $s\in \Gal(E/F)$, are distinct. The following lemma is easy and can be found in \cite[\S 19.1]{B-H-GL2}.
  \begin{lem}\label{lemma::tame pair}
    Let $(E/F,\theta)$ be a tame pair. The following statements are equivalent:    
    \begin{enumerate}
    	\item[(i)] The pair $(E/F,\theta)$ is admissible; 
    	\item[(ii)] The restrictions $\theta^s |_{U_E}$, $s\in \Gal(E/F)$ are distinct;
    	\item[(iii)]  The character $\theta$ cannot factor through $\N_{E/K}$ for any subextension $K/F$.
    \end{enumerate}
  \end{lem}
      
     Two admissible tame pairs $(E_i/F,\theta_i)$, $i=1,2$, are called $F$-isomorphic if there is an $F$-isomorphism $\alpha:E_1 \ra E_2$ such that $\theta_1 = \theta_2 \circ \alpha$. The degree of $(E/F,\theta)$ is the degree $[E:F]$ of the extension $E/F$. Denote by $\mathcal{T}_n(F)$ the $F$-isomorphism classes of admissible tame pairs of degree $n$.
       
     Let $(E/F,\theta)$ be a tame pair of degree $n$. As $\theta$ is trivial on $U^1_E$, it reduces to a character $\tilde{\theta}$ of $U_E/U^1_E \cong \kappa_E^{\times} \cong \BF_{q^n}^{\times}$. The pair $(E/F,\theta)$ is admissible if and only if $\tilde{\theta}$ is $\BF_q$-regular. 
     
     Recall that $\Lambda_n(\BF_q)$ is denoted as the set of equivalence classes of $\BF_q$-regular characters of $\BF_{q^n}^{\times}$. Then we have a bijection
     \begin{align}\label{diagram::tame pair}
     \mathcal{T}_n(F) & \longleftrightarrow \BC^{\times} \times \Lambda_n(\BF_q)    \\
     (E/F,\theta) & \longleftrightarrow (\theta(\varpi), \tilde{\theta}). \nonumber
     \end{align}

     \subsection{Parameterization of level zero representations}

     Firstly, composing \eqref{diagram::tame pair} with Green's parameterization \eqref{diagram::Green} and then the description of level zero supercuspidal representations \eqref{diagram::level 0}, we therefore have a bijective map 
      \begin{align}\label{diagram::level 0 -- tame pair}
     \boldsymbol{\pi}_n \colon \mathcal{T}_n(F) & \longrightarrow \mathcal{A}^0_n(F)_0 \\
      (E/F,\theta) & \longmapsto \boldsymbol{\pi}_n(\theta).  \nonumber 
       \end{align}
     
     Secondly, let $(E/F,\theta) \in \mathcal{T}_n(F)$. We view $\theta$ as a character of $\weil_E$ via Artin's reciprocity map $\boldsymbol{a}_E$, and form the smooth induced representation  
    \begin{align}\label{formula::defn-sigma-theta}
    \boldsymbol{\sigma}_n ( \theta ) = \ind_{\weil_E}^{\weil_F}\shskip\theta=\ind_{E/F}\theta.
    \end{align}
     
     Recall that $\sigma \in \mathcal{G}_n^0(F)$ is called of level zero if $\sigma$ is trivial on the wild inertia subgroup of $\weil_F$ (see \cite{B-H-Explicit+level+0}). Deonte by $\mathcal{G}^0_n(F)_0$ the subset of $\mathcal{G}^0_n(F)$ consisting of classes of representations of level zero. The following result is contained in \cite[A2,A3]{B-H-ETameLLC-I}.
     \begin{prop}
     	Let $(E/F,\theta)$ be an admissible tame pair of degree $n$. Then
     	
     	\begin{enumerate}
     		\item[(1)] The representation $\ind_{E/F} \theta$ is irreducible and of level zero. The equivalence class of $\ind_{E/F} \theta$ depends only on that of $(E/F,\theta)$;
     		\item[(2)] The map
     		\begin{align*}
     		\boldsymbol{\sigma}_n\colon \mathcal{T}_n(F) & \longrightarrow \mathcal{G}_n^0(F)_0    \\
     		(E/F,\theta) & \longmapsto \ind_{E/F} \theta
     		\end{align*}
     		is a bijection.
     	\end{enumerate}
     \end{prop} 
     
     \subsection{Explicit correspondence for level zero supercuspidal representations}
     
     It turns out that $\boldsymbol{\sigma}_n \circ \boldsymbol{\pi}_n^{-1}$ is not the local Langlands correspondence for level zero supercuspidal representations. Some modifications have to be made.
     
     \begin{prop}[\cite{B-H-Explicit+level+0}, Theorem 2]
     Let $(E/F,\theta)$ be an admissible tame pair of degree $n$. Define $\Delta_E$ to be the unique unramified character of $E^{\times}$ of order $2$. Then 
     	\begin{align}\label{formula::explicit level 0}
     	\llc_n^F(\boldsymbol{\pi}_n(\theta)) = \boldsymbol{\sigma}_n (\Delta^{n-1}_E \shskip \theta).
     	\end{align}
     \end{prop}

   \section{Unramified base change}\label{section::Explicit BC}
	
	We determine the unramified base change of a level zero supercuspidal representation in this section.

	 For two unramified extensions $K$ and $E$ over $F$, we will write $KE$ as the compositum of $K$ and $E$, which is still an unramified extension. Recall that $\phi$ is a fixed arithmetic Frobenius element in $\weil_F$. For two nonnegative integers $n$ and $m$, let $(n,m)$ denote as usual the greatest common divisor of $n$ and $m$.
	 
	\begin{prop}\label{prop::base change}
		Let $(E/F,\theta) \in \mathcal{T}_n(F)$ and $\pi=\boldsymbol{\pi}_n(\theta)$ as in \eqref{diagram::level 0 -- tame pair}. Assume that $K/F$ is an unramified field extension of degree $m$. Then 
		\begin{align}\label{formula::base change}
		\pi_{K/F} = \boldsymbol{\pi}_{n/(n,m)}(\varsigma_1) \times \cdots \times \boldsymbol{\pi}_{n/(n,m)}(\varsigma_{(n,m)}), 
		\end{align}
		where, for each $i$, 
		\begin{align}\label{formula::defn-Varsigma i}
		\varsigma_i = \Delta_{KE}^{n/(n,m)-1}\cdot (\Delta_E^{n-1} \shskip \theta)\circ \phi^i|_E \circ \N_{KE/E}
		\end{align}
		is a character of $(KE)^{\times}$ and $(KE/K,\varsigma_i)$ is an admissible tame pair. The representation on the right hand side of \eqref{formula::base change} does not depend on the order of these $\boldsymbol{\pi}_{n/(n,m)}(\varsigma_i)$.
	\end{prop}
	 \begin{proof}
	 	By \eqref{formula::base change and LLC} and \eqref{formula::explicit level 0}, we have
	 	\begin{align}\label{formula::BC 1st}
	 	\llc_n^K(\pi_{K/F}) = \res_{K/F}\llc_n^F(\pi) =  \mathrm{Res}_{K/F}\ind_{E/F}\Delta_E^{n-1}\theta.
	 	\end{align}
	 	As $\mathcal{W}_E$ is a normal subgroup of $\weil_F$ of finite index, we can apply Mackey's restriction formula \eqref{formula::Mackey-restriction}. Note that, as $E/F$, $K/F$ are unramified extensions, the inertia groups $\mathcal{I}_F$, $\mathcal{I}_E$ and $\mathcal{I}_K$ are all the same. Recall that $\Phi$ is a fixed geometric Frobenius element in $\weil_F$, then $\Phi^n$ (resp. $\Phi^m$) is a geometric Frobenius element of $\weil_E$ (resp. $\weil_K$). Therefore we can take $\{\Phi,\cdots,\Phi^{(n,m)}\}$ to be a set of representatives of double cosets $\weil_K
	 	\backslash \weil_F /\weil_E$. Note that $EK$ is the unramified extension of $F$ of degree $[n,m]$, where $[n,m]$ is the least common multiple of $n$ and $m$, and that $\weil_E \cap \weil_K= \weil_{KE}$. Applying \eqref{formula::Mackey-restriction}, we get 
	 	\begin{align}\label{formula::BC-decomposition}
	 	\mathrm{Res}_{K/F}\ind_{E/F}\Delta_E^{n-1}\theta =\mathop{\oplus}_{i=1}^{(n,m)} \ind_{KE/K}    {}^{\Phi^i}\!(\Delta_E^{n-1}\theta)|_{\weil_{KE}}.
	 	\end{align}
        It follows from \cite[Proposition 11, \S 4, Chap XIII]{Serre} that the character ${}^{\Phi^i}\!(\Delta_E^{n-1}\theta)$ of $\weil_E$ corresponds exactly to the character $(\Delta_E^{n-1}\shskip \theta) \circ \phi^i|_E$ of $E^{\times}$ via local class field theory, where $\phi=\Phi^{-1}$ is a arithmetic Frobenius element of $\weil_F$. Hence the restriction of ${}^{\Phi^i}\!(\Delta_E^{n-1}\theta)$ to $\weil_{KE}$ corresponds to the character $(\Delta_E^{n-1} \shskip \theta)\circ \phi^i|_E \circ \N_{KE/E}$ of $(KE)^{\times}$. So we rewrite \eqref{formula::BC-decomposition} as
        \begin{align}\label{formula::BC-decomposition-II}
        \mathrm{Res}_{K/F}\ind_{E/F}\Delta_E^{n-1}\theta =\mathop{\oplus}_{i=1}^{(n,m)} \ind_{KE/K} (\Delta_E^{n-1} \shskip \theta)\circ \phi^i|_E \circ \N_{KE/E}.
        \end{align}
        
        \textbf{Claim}: $(KE/K,(\Delta_E^{n-1} \shskip \theta)\circ \phi^i|_E \circ \N_{KE/E})$ is an admissible tame pair of degree $n/(n,m)$. In fact, $KE/K$ is an unramified field extension of degree $[n,m]/m=n/(n,m)$. Set $\xi=(\Delta_E^{n-1} \shskip \theta)\circ \phi^i|_E \circ \N_{KE/E}$. As $N_{KE/E}(U^1_{KE}) \subset U^1_{E}$ and $\theta$ is trivial on $U^1_E$, $\xi$ is trivial on $U_{KE}^1$. So $(KE/K,\xi)$ is by definition a tame pair of degree $n/(n,m)$.  We are then left to show that $(KE/K,\xi)$ is admissible. Suppose that $\xi$ is not admissible, then it is fixed by a subgroup $G'$, $G'\neq e$, of $\Gal(KE/K)$. Note that $\xi$ is obviously fixed by $\Gal(KE/E)$. As the orders of $\Gal(KE/K)$ and $\Gal(KE/E)$ are coprime, we infer that $\xi$ is fixed by a subgroup of $\Gal(KE/F)$ that is strictly larger than $\Gal(KE/E)$. By the arguments in the proof of Lemma \ref{lemma::tame pair}, this means that $\xi$ factors through $\N_{KE/L}$, where $L/F$ is a subextension of $E/F$. By the transitivity of the norm map and the fact that $\N_{KE/E}$ is surjective, we conclude that $(\Delta_E^{n-1}\shskip \theta)\circ \phi^i|_E$ factors through $\N_{E/L}$. This implies easily that $\theta$ also factors through $\N_{E/L}$, which contradicts the assumption that $(E/F,\theta)$ is an admissible pair.
        
        Therefore, by \eqref{formula::explicit level 0}, we have
        \begin{align}\label{formula::BC 3}
        \ind_{KE/K}(\Delta_E^{n-1} \shskip \theta)\circ \phi^i|_E \circ \N_{KE/E} =     \llc_{n/(n,m)}^K(\boldsymbol{\pi}_{n/(n,m)} (\varsigma_i))
        \end{align}
        with $\varsigma_i$ defined in \eqref{formula::defn-Varsigma i}.
        Note that no two of the representations $\boldsymbol{\pi}_{n/(n,m)}(\varsigma_i)$ and $\boldsymbol{\pi}_{n/(n,m)}(\varsigma_j)$ are linked in the sense of \cite{Zelevinsky-II} and that the condition \eqref{Condition} is satisfied, hence the representation 
        $$\boldsymbol{\pi}_{n/(n,m)}(\varsigma_1) \times \cdots\times \boldsymbol{\pi}_{n/(n,m)}(\varsigma_{(n,m)})$$
         is irreducible (\cite[Theorem 4.2]{Zelevinsky-II}) and does not depend on the order of these $\boldsymbol{\pi}_{n/(n,m)(\varsigma_i)}$. 
         
         By \eqref{formula::LLC-additive}, together with \eqref{formula::BC 3}, \eqref{formula::BC-decomposition-II} and \eqref{formula::BC 1st}, we get
         \begin{align*}
         \llc_n^K(\pi_{K/F})& = \mathop{\oplus}_{i=1}^{(n,m)} \llc_{n/(n,m)}^K(\boldsymbol{\pi}_{n/(n,m)}(\varsigma_i) \nonumber\\
         &=\llc_n^K(\boldsymbol{\pi}_{n/(n,m)}(\varsigma_1) \times \cdots\times \boldsymbol{\pi}_{n/(n,m)}(\varsigma_{(n,m)})) .
         \end{align*}
        So, by the Langlands correspondence, we are done.
        

	 	\end{proof}
	In particular, we see that an unramified base change of $\pi \in \mathcal{A}^0_n(F)_0$ is again supercuspidal if and only if the degree of the base field extension is coprime with $n$.

	
To make use of \eqref{formula::gamma-BC + AI} later in our computation, we need one simple observation. Any level zero supercuspidal representation is automorphically induced by a character.

\begin{lem}\label{lem::level 0 =AI}
	Let $(K/F,\lambdaup)$ be an admissible tame pair of degree $m$ and $\pi_m(\lambda)$ be the representation as in \eqref{diagram::level 0 -- tame pair}. Then
	\begin{align}
	\boldsymbol{\pi}_m (\lambdaup) = (\Delta_K^{m-1}\lambdaup)^{K/F}.
	\end{align}
\end{lem}
\begin{proof}
	This follows immediately from \eqref{formula::automorphic induction} and \eqref{formula::explicit level 0}.
\end{proof}

    \section{Main result}

    Before the proof of our main result, we note that
    \begin{lem}\label{lemma::gamma-character}
    	Let $K/F$ be an unramified field extension of degree $m$. Let $\varXi$ be the group of characters of $F^{\times}$ that is trivial on $\N_{K/F}(K^{\times})$. Suppose that $\psi$ has level one. Then
    	\begin{align}
    	\prod_{\chi \in \varXi} \gamma(s,\chi,\psi) = (-1)^{m-1}\gamma(s,1_K,\psi_K).
    	\end{align}
    \end{lem}
    \begin{proof}
    	The characters in $\varXi$ are unramified as $K/F$ is unramified and the norm map $\N_{K/F}$ is surjective. Then by \cite[\S 23.4,\S 23.5]{B-H-GL2}, for $\chi \in \varXi$, we have that
    	\begin{align}
    	\gamma(s,\chi,\psi) = \frac{q^{s-1/2}}{\chi(\varpi)}\cdot \frac{1-\chi(\varpi)q^{-s}}{1-\chi(\varpi)^{-1}q^{s-1}}=-q^{-s-1/2}\frac{q^s-\chi(\varpi)}{q^{s-1}-\chi(\varpi)}.
    	\end{align}
    	When $\chi$ runs over $\varXi$, the value $\chi(\varpi)$ runs over the m-th roots of unity, as the group $\varXi$ is cyclic of order $m$. Hence
    	\begin{align}
    	\prod_{\chi \in \varXi} \gamma(s,\chi,\psi) = (-1)^mq^{-ms-m/2}\frac{q^{ms}-1}{q^{ms-m}-1}=(-1)^{m-1}\gamma(s,1_K,\psi_K).
    	\end{align}
    	\end{proof}
    \begin{thm}
    	Let $\pi \in \mathcal{A}^0_n(F)_0$ and $\tau \in \mathcal{A}^0_m(F)_0$ with $ n > m$ and $\psi$ of level one. Then
    	\begin{align}\label{formula::Main}
    	\gamma (s,\pi \times \tau,\psi) = (-1)^{nm-(n,m)}q^{-nm/2}\prod_{i=0}^{(n,m)-1} G(\shskip \widetilde{\eta_{\pi}}^{q^i}\circ \N_{[n,m]:n}\cdot \widetilde{\eta_{\tau}}\circ \N_{[n,m]:m},\tilde{\psi}).
    	\end{align}
    \end{thm}
    where $\widetilde{\eta_{\pi}}$ resp. $\widetilde{\eta_{\tau}}$ is the regular character of $\BF_{q^n}^{\times}$ resp. $\BF_{q^m}^{\times}$ that corresponds to $\tilde{\pi}$ resp. $\tilde{\tau}$ via Green's parameterization and $G$ is the Gauss sum as defined by \eqref{defn::Gauss sum}.
   \begin{proof}
   	We split the proof in the following steps.
 

   	 Step (1). Suppose that $(K/F,\eta_{\tau})$ is the admissible tame pair that corresponds to $\tau$ by $\boldsymbol{\pi}_m$ (see \eqref{diagram::level 0 -- tame pair}). By Lemma \ref{lem::level 0 =AI}, we have  
   	 $$\tau = \boldsymbol{\pi}_{m}(\eta_{\tau})=(\Delta_K^{m-1}\eta_{\tau})^{K/F}.$$
   	
   	Applying \eqref{formula::gamma-BC + AI}, we get
   	\begin{align}\label{formula::Main I}
   		\gamma (s, \pi \times \tau,\psi)&=\gamma (s, \pi \times (\Delta_K^{m-1}\eta_{\tau})^{K/F},\psi) \nonumber \\
   		&=\gamma (s,\pi_{K/F} \times \Delta_K^{m-1}\eta_{\tau} ,\psi_K) \frac{\prod_{\chi \in \varXi}\gamma (s,\chi,\psi_F)^{n} }{\gamma (s,1_K,\psi_K)^{n}}\\
   		&=(-1)^{mn-n}\gamma (s,\pi_{K/F} \times \Delta_K^{m-1}\eta_{\tau} ,\psi_K).\nonumber 
   	\end{align}
   	The last equality follows from Lemma \ref{lemma::gamma-character}.

   	Step (2). Suppose that $(K'/F,\eta_{\pi})$ is the admissible tame pair that corresponds to $\pi$. So $\pi = \boldsymbol{\pi}_n(\eta_{\pi})$. By \eqref{formula::base change}, we get
   	\begin{align*}
   \pi_{K/F}= \boldsymbol{\pi}_{n/(n,m)}(\varsigma_1) \times \cdots \times \boldsymbol{\pi}_{n/(n,m)}(\varsigma_{(n,m)}).
   	\end{align*}
   	where $\boldsymbol{\pi}_{n/(n,m)}(\varsigma_i)\in \mathcal{A}^0_{n/(n,m)}(K)_0$ corresponds to the admissible tame pair $(K'K/K,\varsigma_i)$, and 
   	\begin{align}\label{formula::Varsigma_i}
   	\varsigma_i = \Delta_{KK'}^{n/(n,m)-1}\cdot (\Delta_{K'}^{n-1} \shskip \eta_{\pi})\circ \phi^i|_{K'} \circ \N_{KK'/K'}.
   	\end{align}
   	 By the multiplicativity of local $\gamma$-factors \cite[Theorem 3.1]{J-PS-S-Rankin+Selberg}, we get
   	\begin{align}\label{formula::Main II}
   	\gamma (s,\pi_{K/F} \times \Delta_K^{m-1}\eta_{\tau} ,\psi_K) = \prod_{i=1}^{(n,m)} \gamma(s,\boldsymbol{\pi}_{n/(n,m)}(\varsigma_i)\times \Delta_K^{m-1}\eta_{\tau} ,\psi_K).
   	\end{align}
   	
   	Step (3). Now we invoke the formula \eqref{formula::Connection finite-local}. We have, for each $i$, that
   	\begin{align}\label{formula::Main III}
   	&\gamma(s,\boldsymbol{\pi}_{n/(n,m)}(\varsigma_i)\times \Delta_K^{m-1}\eta_{\tau} ,\psi_K)\nonumber \\&= \eta_{\tau}(-1)^{n/(n,m)-1}q^{[n,m]/2-m}\gamma(\widetilde{\boldsymbol{\pi}_{n/(n,m)}(\varsigma_i)}\times \widetilde{\Delta_K^{m-1}\eta_{\tau}},\widetilde{\psi_K}).
   	\end{align}
   	The gamma factor on the right hand side of \eqref{formula::Main III} is a gamma facotr over finite fields. From \eqref{formula::Varsigma_i}, the regular character which corresponds to $\widetilde{\boldsymbol{\pi}_{n/(n,m)}(\varsigma_i)}$ is 
   	\begin{align*}
   	\widetilde{\eta_{\pi}}^{q^i}\circ \N_{[n,m]:n}.
   	\end{align*}
   	 In view of \eqref{formula::Nien n*1}, we get
   	\begin{align}\label{formula::Main IV}
   	&\gamma(\widetilde{\boldsymbol{\pi}_{n/(n,m)}(\varsigma_i)}\times \widetilde{\Delta_K^{m-1}\eta_{\tau}},\widetilde{\psi_K}) \nonumber\\&= (-q^{-m}\widetilde{\eta_{\tau}}(-1))^{n/(n,m)-1}G(\widetilde{\eta_{\pi}}^{q^i}\circ \N_{[n,m]:n}\cdot \widetilde{\eta_{\tau}} \circ \N_{[n,m]:m},\widetilde{\psi}).
   	\end{align}
   	
   	Finally, combining \eqref{formula::Main I}, \eqref{formula::Main II}, \eqref{formula::Main III} and \eqref{formula::Main IV} together, we get \eqref{formula::Main}.

   	\end{proof}
    We can also reformulate a formula for gamma factors over finite fields.
    \begin{thm}\label{theorem::Main-Gamma-finite}
    	Let $\pi$ and $\tau$ be irreducible cuspidal representations of $\GL_n(\BF_q)$ and $\GL_m(\BF_q)$ respectively, $n > m$. Let $\eta_{\pi}$ and $\eta_{\tau}$ be the correspoding $\BF_q$-regular characters of $\BF_{q^n}^{\times}$ and of $\BF_{q^m}^{\times}$ via Green's parameterization. Then
    	\begin{align}
    	\gamma(\pi\times \tau,\psi) = (-1)^{nm-(n,m)}\eta_{\tau}(-1)^{n-1}q^{-mn+\tfrac{m^2+m}{2}}\prod_{i=0}^{(n,m)-1} G(\shskip \eta_{\pi}^{q^i}\circ \N_{[n,m]:n}\cdot \eta_{\tau}\circ \N_{[n,m]:m},\psi). 
    	\end{align}
    \end{thm}

\appendix

\section{Mackey's restriction formula}

\begin{prop}
	Let $G$ be a locally profinite group. Let $H$ be an open subgroup, and let $(\sigma,W)$ be a smooth representaton of $H$. Let $K$ be a closed subgroup of $G$. There is a natural isomorphism
	\begin{align}\label{formula::Mackey-restriction}
	\res^G_K \, \cind_H^G \sigma \cong  \mathop{\bigoplus}\limits_{g \in K \backslash G /H} \cind_{K \cap \prescript{g}{}{H}}^K \res^{\prescript{g}{}{H}}_{K \cap \prescript{g}{}{H}}  \prescript{g}{}{\sigma}
	\end{align}
	where $\prescript{g}{}{H}$ is the subgroup $gHg^{-1}$ of $G$ and $\prescript{g}{}{\sigma}$ is the representation of $\prescript{g}{}{H}$ on $W$ defined by $\prescript{g}{}{\sigma}(ghg^{-1}) = \sigma(h)$.
\end{prop}

    \begin{proof}
	Suppose $\{g_i\}$ is a set of representative of double cosets $K\backslash G / H$ and $\{k^{(i)}_j\}$ is a set of representative of cosets $K/K\cap g_i H g_i^{-1}$. Then $\{k^{(i)}_jg_i\}$ is a set of representative of cosets $G/H$. For $w\in W$, let $f_w \in \cind_H^G \sigma$ such that $f_w$ is supported in $H$ and $f_w(h) = \sigma(h) w$. Suppose $\BW$ is a $\BC$-basis of $W$. Then, by the lemma in \cite[2.5]{B-H-GL2}, $\{k^{(i)}_jg_i f_w \ |\ w\in \BW\}$ is a $\BC$-basis of $\cind_H^G \sigma$. For a fixed $g_i=g$, we show that $\{k_jgf_w \ |\ w\in \BW \}$ is a $\BC$-basis of a $K$-representation that is isomorphic to $\cind^K_{K \cap \prescript{g}{}{H}}\mathrm{Res}^{\prescript{g}{}{H}}_{K \cap \prescript{g}{}{H}} \prescript{g}{}{\sigma}$. This latter representation has a $\BC$-basis $\{k_j\phi_w \ |\ w\in \BW \}$, where $\phi_w$ is supported in $K \cap \prescript{g}{}{H}$ and $\phi_w(k) = \sigma(g^{-1}kg)w$. Then the map $k_jgf_w \mapsto k_j\phi_w$ extends linearly and gives the required $K$-isomorphism. In fact, $k_j\phi_w = \lambda(g) k_jgf_w |_K$, where $\lambda(g)$ is the left translation by $g$, so the $K$-equivalence follows immediately. 
    \end{proof}


	\bibliographystyle{alphanum}
	\bibliography{references}
	
\end{document}